\documentclass[11pt, reqno]{amsart}
\usepackage[margin=1in]{geometry}
\usepackage{graphicx}
\usepackage{relsize}
\usepackage{amsmath}
\usepackage{xcolor}
\usepackage{braket}
\newcommand\numberthis{\addtocounter{equation}{1}\tag{\theequation}}

\numberwithin{equation}{section}
\theoremstyle{plain}
\newtheorem{thm}{Theorem}[section]
\newtheorem*{thm*}{Theorem}

\newtheorem{lem}[thm]{Lemma}

\newtheorem{cor}{Corollary}
\newtheorem*{cor*}{Corollary}
\theoremstyle{definition}

\theoremstyle{remark}

\newtheorem{case}{Case}

\newcommand{\BigOb}[1]{\ensuremath{\operatorname{O}_b\left(#1\right)}}

\newcommand{\BigO}[1]{\ensuremath{\operatorname{O}\left(#1\right)}}

\makeatletter
\def\pmod#1{\allowbreak\mkern10mu({\operator@font mod}\,\,#1)} % Shortens \pmod space
\makeatother

\begin{document}

 \title{Mean value estimates of gcd and lcm-sums}

\author{Sneha Chaubey}
\email{sneha@iiitd.ac.in}

\author{Shivani Goel}
\email{shivanig@iiitd.ac.in}

\address[]{Department of Mathematics, IIIT Delhi, New Delhi 110020}
%{Department of Mathematics, IIIT Delhi, New Delhi 110020}

\begin{abstract}
We study the distribution of the generalized gcd and lcm functions on average. The generalized gcd function, denoted by $(m,n)_b$, is the largest $b$-th power divisor common to $m$ and $n$. Likewise, the generalized lcm function, denoted by $[m,n]_b$, is the smallest $b$-th power multiple common to $m$ and $n$. We derive asymptotic formulas for the average order of the arithmetic, geometric, and harmonic means of $(m,n)_b$. Additionally, we also deduce asymptotic formulas with error terms for the means of $(n_1,n_2,\cdots, n_k)_b$, and $[n_1,n_2,\cdots, n_k]_b$ over a set of lattice points, thereby generalizing some of the previous work on gcd and lcm-sum estimates.

%We study the distribution of the generalized gcd function on average. Following E. Cohen, denoted by $(m,n)_b$, it is the largest $b$-th power divisor common to $m$ and $n$. We derive asymptotic formulas for the average order of the arithmetic, geometric, and harmonic means of $(m,n)_b$. 
\iffalse
We consider the greatest $b^{th}$ power common divisor of two numbers and {\cite{Madhu}}  say $\gcd_b $ and then $\gcd_b-sum$ where $\gcd_b-sum$ is an arithmetic function defined as $h_b(n)= \sum_{j\leq n}{\gcd_b(j,n)}$. We show that this function is multiplicative, and its Dirichlet series has a compact representation in terms of Riemann zeta function. At the end, we also establish asymptotic formulas for its average and weighted average.
\fi
\end{abstract}

\maketitle
\section{Introduction and Main Results}
The study of the greatest common divisor (gcd) function arises naturally in number theory, and several mathematicians over the years have been interested in its distribution, and other arithmetic and algebraic properties. See \cite{Toth} for a survey article on gcd-sums. The most natural geometric usage of the gcd function is in counting visible lattice points.
Recently, a new notion of visibility has gained much attention, where curves replace straight lines as the lines of sight \cite{Visibility, Harris}. The immediate natural extension is to consider visibility along curves $y=kx^b$ or $x=ky^b$, where $b\in \mathbb{N}$ and $k\in \mathbb{Q}$.  %Recall that a visible lattice point $(r,s)\in\mathbb{Z}^2$ is a pair of integers such that the line segment joining $(r,s)$ to origin does not contain any other lattice point between them. 
A point $(r,s)\in \mathbb{Z}^2$ is visible along the the curve $y=kx^b$ if and only if
$\gcd_b(r,s):=\max\{d\geq 1: d|r\ \text{and}\  d^b|s\}=1$, and similarly, a point $(r,s)\in \mathbb{Z}^2$ is visible along the the curve $x=ky^b$ if and only if
${\gcd'}_b(r,s):=\max\{d\geq 1: d^b|r\ \text{and}\  d|s\}=1$.
%For a fixed $b\in\mathbb{N}$, a point $(r,s)\in\mathbb{Z}^2$ is said to be $b_y$-visible {\cite{Goins}}, if it lies on the curve $y=kx^b$ for $k\in\mathbb{Q}$, and there is no other lattice point on the curve between the origin and $(r,s)$. Equivalently, $(r,s)\in \mathbb{Z}^2$ is $b_y$-visible if and only if $\gcd_{b_y}(r,s):=\max\{d\geq 1: d|r\ \text{and}\  d^b|s\}=1$. Likewise, for the curves $x=ky^b$ for $k\in\mathbb{Q}$, one can define $b_x$-visibility: $(r,s)$ is $b_x$-visible if and only if $\gcd_{b_x}(r,s):=\max\{d\geq 1: d^b|r\ \text{and}\  d|s\}=1$. 
The distribution of these extended notions, ${\gcd}_b$, ${\gcd'}_b$ of the classical gcd function have been studied before, for example in \cite{Visibility}, where the authors study various mean value results for ${\gcd}_b(r,s)$. 

A more natural generalization of the usual gcd function is
         \[(r,s)_b:=\max\{d\geq 1: d^b|r\ \text{and}\  d^b|s\} \numberthis\label{main},\] 
         which appears 
        when studying the $b$-free integers. It also provides a necessary criterion for the combined visibility of a lattice point along the curves $y=kx^b$ or $x=ky^b$.
       % or with respect to the generalized lattice visibility as a necessary criterion. It is defined as 
         %\[(r,s)_b:=\max\{d\geq 1: d^b|r\ \text{and}\  d^b|s\}. \numberthis\label{main}\]
         %Moreover, if an integer lattice point $(r,s)$ is ${b_y}$-visible or ${b_x}$-visible then $(r,s)_b=1$. 
         Following Cohen \cite{Cohen}, this gcd function is called the generalized gcd function. For $b=1$, it is the usual gcd function. It has most commonly been used to study the generalized totient function or commonly called the Klee's totient function \cite{Klee}. Furthermore, we define the generalized least common multiple (lcm) function by
         \[[r,s]_b:=\min\{d\ge 1: r|d^b \ \text{and} \ s|d^b\}. \numberthis\label{mainlcm}\] Unlike, the case for the usual gcd, there is no direct relation between $(r,s)_b$ and $[r,s]_b$ for $b\ge 2$ but it has several interesting properties which nicely generalizes the usual lcm function. 
        
         In this note, we are interested in studying the distribution of \eqref{main} by estimating the average order of its arithmetic, geometric, and harmonic means. We also obtain asymptotic formulas for the average order of the mean of generalized lcm \eqref{mainlcm} defined over $k$ variables.   
         
        The arithmetic mean, or the generalized gcd-sum function is defined as follows: for a fixed integer $b>1$,
         \[h_b(n):= \sum_{j\leq n}{(j,n)_b}. \numberthis\label{sumfunction}\]
         Note that this is a multiplicative function (proof in Section 2), whereas the analogous summatory functions for ${\gcd}_b$, ${\gcd'}_b$ are not multiplicative, and thus it becomes an interesting object of study from an analytic point of view. 
         
         The case $b=1$, also known as Pillai's arithmetical function defined as
         \[P(n)=\sum_{j\le n}\gcd(j,n), \numberthis\label{Pillai}\]
         has been explored before. Toth in his survey article \cite{Toth} on gcd-sum functions estimated the average order of $P(n)$ to be: for every $\epsilon>0$, 
         \[\sum_{n\le x}P(n)=\frac{1}{2\zeta(2)}x^2\log x+\left(2\gamma-\frac{1}{2}+\frac{\zeta'(2)}{\zeta(2)}\right)x^2+\BigO{x^{1+\theta+\epsilon}},\]
         where $\gamma$ is the Euler's constant and $\theta$ as in
         \[\sum_{n\le x}d(n)-x\log x-(2\gamma-1)x=\BigO{x^{\theta+\epsilon}}.\]
         Estimates of this type but with weaker error terms were obtained before by some authors; see  \cite{Bord}, \cite{Broughan}, and \cite{Kopetzky}.
         %Many generalizations of the gcd-sum function have been considered including investigating for an arbitrary arithmetical function $f$, that is,
         The proofs use the convolution identity of Cesaro \cite{Cesaro} given by
         \[\sum_{k=1}^jf(\gcd(k,j))=(f\ast\phi)(k),\] for any arithmetical function $f$ and $\ast$ denotes the Dirichlet convolution. It is natural to ask if similar estimates hold true for the mean of generalized gcd function: $h_b(n)$. Using Perron's formulas and results on the average order of Klee's totient function, we establish asymptotics of the mean and the average version of the mean in the following theorems below.  
         %Now, we define $\gcd_b$-sum function  as 
         
         %The function $h_b(n)$ is multiplicative, and the corresponding Dirichlet series converges in the whole complex plane, except the zeroes of $\zeta(2s)$ and the points $s=2/b$ and $s=2$ which are the poles of Dirichlet series.
         
         %With the help of Perron formula, we establish a relation between asymptotic averages of h(n) and its Dirichlet series and then evaluate averages. %
\begin{thm}\label{TheoremOne}
Let $b=\{2,3\}$, there exist a constant $c_b>0$ such that
\[\sum_{n\leq x}\left(1-\frac{n}{x}\right)h_b(n)=\frac{\zeta(2b-1)}{6\zeta(2b)}x^2+\frac{b^2\zeta(2/b-1)}{2(2+b)\zeta(2)}x^{2/b}+ \text{O}_b\left(\frac{x^{2/b}\left(\log x\right)^{6/5}}{e^{c_b(\log x)^{3/5} (\log \log x)^{-1/5}}}\right).
\]
\end{thm}
In the case when $b>3$, we prove
\begin{thm}\label{TheoremTwo}
For $b>3$,
\[\sum_{n\leq x}\left(1-\frac{n}{x}\right)h_b(n)= \frac{\zeta(2b-1)}{6\zeta(2b)}x^2 +\text{O}_b(x\log x).\]
\end{thm}
\begin{thm}\label{TheoremThree} 
For $b\ge 2$, 
\[\sum_{n\leq x}h_b(n)= \frac{\zeta(2b-1)}{2\zeta(2b)}x^2+\text{O}_b(E_b(x)),\]
where $E_2(x)=\BigO{x\log x}$, and $E_b(x)=\BigO{x}$ for $b\ge 3.$
\end{thm}

The harmonic mean of the generalized gcd function is given by
\[H_b(n)=n\left(\sum_{j\le n}\frac{1}{(j,n)_b}\right)^{-1}. \numberthis\label{harmonicmean}\] For the usual gcd function, the average order of its harmonic mean is given by 
\[\sum_{n\le x}\left(\sum_{j\le n}\frac{1}{\gcd(j,n)}\right)^{-1}=c_1\log x+c_2+\BigO{x^{-1+\epsilon}},\]
 for any $\epsilon>0$, and constants $c_1, c_2>0$ \cite{Toth}. The harmonic mean for the usual gcd function has an algebraic significance in terms of counting generators of a cyclic group of order $n$, see \cite{Toth}, and the references therein. One obtains a similar result for $H_b(n)/n$.  
\begin{thm}\label{TheoremFive} 
For $b\ge 2$, and $H_b(n)$ as in \eqref{harmonicmean}, we have
\[\sum_{n\leq x}\frac{H_b(n)}{n}= C_b \log x +\BigOb1,\]
where $C_b=\dfrac{1}{\zeta(b)}\mathlarger{\prod}_p\left[1+(p^{b+1}-1)\mathlarger{\sum}_{k=1}^{\infty}\dfrac{p^k}{p^{kb+1}(p^{b+1}-p)+p-1}\right].$
\end{thm}
Furthermore, we consider the geometric mean of the generalized gcd function. For $b\ge 2$, it is
\[G_b(n):=\left(\prod_{j\le n}(j,n)_b\right)^{1/n}. \numberthis\label{Geometricmean}\]
Geometric mean for the classical $\gcd$-sum function was considered by Loveless in \cite{loveless}. 
\begin{thm}\label{TheoremSix}
For $b\ge 2$, we have
\[\sum_{n\le x}n\log G_b(n)=-\frac{\zeta'(2b)}{2\zeta(2b)}x^2+\text{O}(x).\]
%which implies \[\prod_{n\le x}g_b(n)=\text{O}(x^{e^{1/b+\epsilon}})\exp(\frac{x^2\zeta'(2b)}{2\zeta(2b)})\]
\end{thm}
Let $k\ge 2$, and $f$ be any arithmetic function. We consider the weighted average order of the generalized gcd function over a set of lattice points given by
\[M_{f,b,k}(x):=\sum_{n_1, n_2, \cdots, n_k\le x}f((n_1, n_2, \cdots, n_k)_b), \numberthis\label{doublegengcd}\]
where $(n_1, n_2, \cdots, n_k)_b=\max\{d\ge 1: d^b|n_1, \cdots, d^b|n_k\}.$
%\[M_b(x):=\sum_{m,n\le x^{b}}(m,n)_{b}. \numberthis\label{doublegcd}\]
In other words, if $N_1, N_2, \cdots, N_k$ are random integers chosen uniformly and independently from the set $\{1,2, \dots, x\},$ then the above gives the weighted expected value of $(N_1, N_2, \cdots, N_k)_b$. For $k=2, b=1$, and $f=\text{id}$, \eqref{doublegengcd} was studied by Cesaro \cite{Cesaro}, and later by Diaconis and Erdos \cite{Erdos} who improved the error term obtained in \cite{Cesaro}. Similar to the usual gcd-sum function, for $f=\text{id}$, $k=2$, the common technique to study \eqref{doublegengcd}, is using its relation with the summatory function \eqref{sumfunction}. For $k\ge 3$, one can use the identity
\[M_{f,b,k}(x)=\sum_{d\le x^{1/b}}(f\ast\mu)(d)\left \lfloor{\frac{x}{d^{1/b}}}\right \rfloor^k,\] 
and derive the following results.
\begin{thm}\label{TheoremFour} Let $M_{f,b,k}(x)$ be defined by \eqref{doublegengcd}, then for $f=\text{id}$, we have
\[M_{\text{id},2,2}(x)= x^2\frac{\zeta(3)}{\zeta(4)}+\text{O}(x\log x).\]
Moreover, for $b\geq 3, k\ge 2$, \[M_{\text{id},b,k}(x)=\frac{\zeta(bk-1)}{\zeta(bk)}x^k+\BigOb{x^{k-1}}.\]
\end{thm}
%\begin{thm}\label{TheoremSeven}For $b\ge 2$ and $k\ge 2$, we have the following asymptotic formula: \[\sum_{n_1,n_2,...,n_k\le x}(n_1,n_2,...,n_k)_b=\frac{\zeta(bk-1)}{\zeta(bk)}x^k+\text{O}(x^{k-1})\]
%\end{thm}
For $b=1$, the lcm-sum function for $2$ variables can be studied from the gcd-sum employing the relation $\text{lcm}(m,n)\gcd(m,n)=mn$, and an asymptotic formula for the same was established in \cite{Erdos}. For any positive real number $r$, it was proved that
\[\sum_{n_1,n_2\le x}\left(\text{lcm}(n_1,n_2)\right)^r=\frac{\zeta(r+2)}{(r+1)^2\zeta(2)}x^{2(r+1)}+\BigO{x^{2r+1}\log x}.\]  
A better error term can be obtained for $r\in\mathbb{N}$ using Walfisz \cite{Walf} estimate on the average order of the Euler-phi function. For the generalized lcm defined in \eqref{mainlcm}, one does not have a direct relation with the generalized gcd \eqref{main} unlike for $b=1$, and so one cannot use Theorem \ref{TheoremFour} directly. Instead, we follow the method of \cite{Hilber}, and obtain asymptotic formulas for the mean values of generalized lcm function. As in \cite{Hilber}, the method can be extended for $k\ge 2$, and for a wider class of multiplicative arithmetic functions. Following notations as in \cite{Hilber}, for $r\in\mathbb{R}$, we define the class $\mathcal{A}_r$ of complex valued arithmetic functions satisfying the following properties: A function $f\in \mathcal{A}_r$ if there exists constants $C_1$ and $C_2$ such that 
%and from above two equations
%\[|f(p^v)|\le C_3p^{vr}\ \text{for all prime power}\ p^v \text{with}\ v\ge 1\]
%where $C_3= \max \{C_1+1,C_2\}$.
\begin{enumerate}
    \item $|f(p)-p|\le C_1p^{r-1/2}$ for all primes $p$.
    \item $|f(p^v)|\le C_2p^{vr}$ for all prime powers $p^v$ with $v\ge 2$.
\end{enumerate}
Let  
\[[n_1, \cdots, n_k]_b=\min\{d\ge 1: n_1|d^b, \cdots, n_k|d^b\}. \numberthis\label{kgenlcm}\] denote the generalized lcm of k variables, then
\begin{thm}\label{TheoremEight}
For $b\ge 2$, $k\ge 2$, $f \in \mathcal{A}_r$, $r\in\mathbb{R},$ and for every $\epsilon>0$, we have
\[\sum_{n_1,n_2,\cdots,n_k\le x}f([n_1, \cdots , n_k]_b)=C^b_{f,k}\frac{x^{k(r+1)}}{(r+1)^k}+\BigOb{x^{k(r+1)-\frac{1}{2}\min(r+1,1)+\epsilon}},\]
where
\[C^b_{f,k}=\prod_p\left(1-\frac{1}{p}\right)^k\sum_{v_1,...,v_k=0}^\infty \frac{f(p^{\lceil\frac{\max(v_1,\cdots,v_k)}{b}\rceil})}{p^{(r+1)(v_1+\cdots+v_k)}}.\]
\end{thm}
The above generalizes the result obtained in \cite{Hilber} by substituting $b=1$. For $k=2$, the constant can be simplified and we obtain
\begin{cor}\label{CorOne}
Let $k=2$, $f= (id)^r$ and $r>-1$ be a real number, then for every $\epsilon>0$, we have
\[\sum_{n_1,n_2\le x}[n_1,n_2]_b^r=C^b_{I_r,2}\frac{x^{2(r+1)}}{(r+1)^2}+\text{O}(x^{2(r+1)-\frac{1}{2}\min(r+1,1)+\epsilon})\]
where
\begin{align*}
    C^b_{I_r,2}&=\zeta^2(r+1)\zeta((r+1)b-r)\zeta(2(r+1)b-r)\prod_p\left(1-\frac{1}{p}\right)^2\left[\frac{1}{p}\left(1-\frac{1}{p^r}\right)\right.\\
&\left.\left(2-\frac{2}{p^{2(r+1)b-r}}-\frac{1}{p^{r+1}}+\frac{1}{p^{(r+1)b+1}}\right)+\left(1-\frac{1}{p^{2(r+1)b-r}}\right)\left(1-\frac{1}{p^{(r+1)b-r}}\right)\right].
\end{align*}
\end{cor}
Section 2 contains some arithmetic and analytic properties of the generalized gcd function, and its arithmetic mean $h_b(n)$. Section 3 contains proofs of Theorems \ref{TheoremOne}, \ref{TheoremTwo}, and \ref{TheoremThree}. Section 4 discusses results on the harmonic mean. Section 5 contains the proof of the average order of the geometric mean. We prove Theorem \ref{TheoremSix} in Section 6. Results on the generalized lcm-sums are discussed in Section 7 of the paper. Finally, in Section 8, we discuss some possible future work emerging out of this note. 
%\end{rem}
\section{Preliminary properties}
The generalized-gcd function, and its mean $h_b(n)$ has some interesting properties. In what follows we collect some of these properties. All of these can be verified using elementary techniques. 
\begin{enumerate}
\item The generalized-gcd is a periodic function. That is, for any $m, r, s\in\mathbb{Z}$, we have: 
\[(r,s)_b=(r+ms,s)_b\] 
    \item The relation between the usual $\gcd$ function and the generalized-$\gcd$ function is given by:
\[ \gcd(r,s)= (r,s)_b\gcd\left(\frac{r}{(r,s)_b},\frac{s}{(r,s)_b}\right).\]
\item For $n=p^{2k}$, where $p$ is prime and $k$ is any positive integer, $h_b(n)$ can be interpreted in the theory of finite Abelian groups in the following manner: \[ h_b(n)=\#\{(i,j)\in\mathbb{Z}_n \times \mathbb{Z}_{p^k} \ | \ i.j^b=0\}.\]
\end{enumerate}
In addition to the above properties $h_b(n)$ owns the following properties which will be used in proving the main results on the average order of $h_b(n)$.
\begin{lem}\label{LemmaTwo}
\begin{enumerate}
\item The function $h_b(n)$ is multiplicative, that is, for $(m,n)=1$, $h_b(mn)=h_b(m)h_b(n)$.
\item For any prime p and for $0 \leq j<b$, $k\in\mathbb{N}$:
    \[h_b(p^{kb+j})=p^{k+j}\left(1+\frac{(p^b-1)(p^{k(b-1)}-1)}{p(p^{b-1}-1)}\right).\]
\item The Dirichlet series associated with $h_b(n)$ is 
\[\sum_{n\geq1}\frac{h_b(n)}{n^s} = \frac{\zeta(bs-1)\zeta(s-1)}{\zeta(2s)}, \]
and is absolutely convergent in $\Re(s)>2$.

\item The function $h_b(n)$ is bounded by \[ |h_b(n)|= \left\{\begin{array}{ll}
         \text{O}(n\log n) & \mbox{if } b=2  \\
         \text{O}(n) & \mbox{if }b\geq 3.
         \end{array}\right.\]
\end{enumerate}
\end{lem}
\begin{proof}
\begin{enumerate}
\item To show $h_b(n)$ is multiplicative, we write it as a convolution of $\phi_b(n)$ and ${n}^{1/b}\chi_b(n)$, where $\chi_b(n)$ is a characteristic b-th power function given by 
\[\chi_b(n)=\left\{\begin{array}{ll}
         1 & \mbox{if } n=m^b\ \text{for some}\ m\in\mathbb{Z},  \\
         0 & \text{otherwise}.
         \end{array}\right.\numberthis \label{chib}\]
         The function $\phi_b(n)$ is the generalization of the Euler phi function attributed to Klee \cite{Klee} and also called as Klee's totient. It is clear that $\phi_b(n)=\phi(n)$ for $b=1$, where $\phi(n)$ is the usual Euler phi function. For a positive integer $b$ and $n$, it is defined as $\phi_b(n):= \#\{k\leq n:(k,n)_b=1 \}$. We observe that $\phi_b$ is a multiplicative function and it is equal to \[ \phi_b(n)= n\prod_{\substack{ p \ \text{prime} \\ p^b|n}}(1-p^{-b}).\numberthis \label{phib}\]
        If $(m,n)_b=d$ then $(\frac{m}{d^b},\frac{n}{d^b})_b=1$, then 
        \begin{align*}\label{hb}
        h_b(n)  & =  \sum_{j\leq n}(j,n)_b
   =\sum_{d^b|n}d\sum_{\substack{j\leq n \\ (j,n)_b=d}}1 =\sum_{d^b|n}d\sum_{\substack{j\leq n \\ (\frac{j}{d^b},\frac{n}{d^b})_b=1}}1
  \\&=\sum_{d^b|n}d\phi_b(\frac{n}{d^b})=\sum_{m|n}{m}^{1/b}\phi_b(\frac{n}{m})\chi_b(m).\numberthis  
         \end{align*} 
         
         The result follows from the multiplicativity of $\phi_b(n)$ and
         ${n}^{1/b}\chi_b(n)$.
          \item From above, we have
         \[h_b(n)=\sum_{m|n}{m}^{1/b}\phi_b(\frac{n}{m})\chi_b(m)=n^{1/b}\sum_{m|n}\frac{1}{{m}^{1/b}}\phi_b(m)\chi_b(\frac{n}{m})\]
         take $n=p^{kb+j}$, $0 \leq j<b$, $k\in\mathbb{N}$
         \begin{align*}
             h_b(p^{kb+j}) & = p^{k+j/b} \left[\frac{p^j}{p^{j/b}}+\sum_{l=1}^k\frac{\phi_b(p^{lb+j})}{p^{l+j/b}}\right]\\
             & = p^{k} \left[p^j+\sum_{l=1}^{k}\frac{p^{lb+j-b}(p^b-1)}{p^l}\right]\\
             & = p^{k+j}\left[1+\frac{(p^b-1)(p^{k(b-1)}-1)}{p(p^{b-1}-1)}\right]
         \end{align*}
\item The function $\phi_b(n)$ has a Dirichlet series associated to it which converges absolutely for $\Re(s)>2$ and  and
has an analytic continuation to a meromorphic function defined on the whole complex plane. It is given by 
\begin{align*}\label{Dirichlet-phi}
    \sum_{n\geq1}\frac{\phi_b(n)}{n^s}& =\prod_p\left(1+\sum_{m=1}^{\infty}\frac{\phi_b(p^m)}{p^{ms}}\right)
 =\prod_p\left(\frac{(1-\frac{1}{p^{bs}})}{(1-\frac{1}{p^{s-1}})}\right)
 \\&= \frac{\zeta(s-1)}{\zeta(bs)}.\numberthis
\end{align*}
The Dirichlet series of $n^{1/b}\chi_b(n)$ converges for $\Re(s)>2/b$, and is given by
\begin{align*}\label{Dirichlet-chi}
    \sum_{n\geq1}\frac{n^{1/b}\chi_b(n)}{n^s} & = \prod_{p}\left(1+\sum_{m=1}^{\infty}\frac{{(p^{m/b}\chi_b(p^m))}}{p^{ms}}\right)
%= \prod_p\left[1+\frac{p}{p^{bs}}+\frac{p^2}{p^{2bs}}+......\right]
 =\prod_p\left(1+\frac{1}{p^{bs-1}}+\frac{1}{p^{2bs-2}}+.......\right)
 %\prod_p\left[\frac{1}{1-\frac{1}{p^{bs-1}}}\right]%
 \\&=\zeta(bs-1).\numberthis\\
\end{align*}
Using \eqref{hb}, \eqref{Dirichlet-phi}, and \eqref{Dirichlet-chi}, we obtain the Dirichlet series of $h_b(n)$ which converges for $\Re(s)>2.$
\[\sum_{n\geq1}\frac{h_b(n)}{n^s} = \frac{\zeta(bs-1)\zeta(s-1)}{\zeta(bs)}.\]
\item The function $|h_b(n)|$ is bounded by
\begin{align*}
    |h_b(n)|& = \left|\sum_{j\leq n}(j,n)_b\right|
=\left|\sum_{d\leq n}d\sum_{\substack{j\leq n \\  (j,n)_b=d}}1\right|
    =\left|\sum_{d\leq n}d\sum_{j\leq n}\sum_{k|\left(\frac{j}{d^b},\frac{n}{d^b}\right)_b}\mu(k)\right|\\
    &\leq \sum_{d\leq n}d\sum_{j\leq n}\sum_{k|\left(\frac{j}{d^b},\frac{n}{d^b}\right)_b}|\mu(k)|
   \leq \sum_{d\leq n}d\sum_{j\leq n}\sum_{\substack{k^b|\frac{j}{d^b} \\  k^b|\frac{n}{d^b}}}1\\
   & \text{Let}\ mk^b= \frac{j}{d^b}  \text{then,} 
m\leq \frac{n}{k^bd^b}\\
&= \sum_{d\leq n^{1/b}}d\sum_{k\leq\frac{n^{1/b}}{d}}\sum_{m\leq\frac{n}{k^bd^b}}1
=\text{O}\left(\sum_{d\leq n^{1/b}}d\sum_{k\leq\frac{n^{1/b}}{d}}\frac{n}{k^bd^b}\right)
=\text{O}\left(\sum_{d\leq n^{1/b}}\frac{n}{d^{b-1}}\sum_{k\leq\frac{n^{1/b}}{d}}\frac{1}{k^b}\right)\\
&=\text{O}\left(\sum_{d\leq n^{1/b}}\frac{n}{d^{b-1}}\right)
   = \left\{\begin{array}{ll}
         \text{O}(n\log n) & \mbox{if } b=2  \\
         \text{O}(n) & \mbox{if }b\geq 3.
         \end{array}\right.
\end{align*}
\end{enumerate}
\end{proof}
\section{Arithmetic Mean}
\subsection{Proofs of Theorems \ref{TheoremOne} and \ref{TheoremTwo}}
%We first lay the general groundwork for the proofs of Theorems \ref{TheoremOne}, and \ref{TheoremThree} and then prove the theorems taking specific values. 
We fix a positive integer $b> 0$, and use Perron's formula to estimate the weighted mean of $h_b(n)$. 
For a complex number $s$, we write $s=\sigma+it$.
By Perron's formula (\cite[page 130]{Tenan}), for $\alpha>0$,
\begin{align}\label{lineintegral1}
    \sum_{n\leq x}\left(1-\frac{n}{x}\right)h_b(n)= \frac{1}{2\pi i} \int_{\alpha-i\infty}^{\alpha+i\infty}{\frac{x^s\zeta(bs-1)\zeta(s-1)}{s(s+1)\zeta(bs)}}ds.
\end{align}
Fix $T,U>0$ such that $2\le T\le x$, $T\le U\le x^2$. Let
\[{\alpha=2+\frac{B}{\log X}}\ \text{and} \ \beta =\frac{1}{b}-\frac{A}{{(\log 2T)}^{2/3}{(\log \log 2T)}^{1/3}},\] 
for some absolute constant $B>0$, and $A=B/b$.
We solve the line integral in \eqref{lineintegral1} by modifying the path of integration from $\alpha-iU$ to $\alpha+iU$ into a path:
$l_1$: $\alpha+iU$ to $\alpha+i\infty$; $l_2$: $2/b+iU$ to $\alpha+iU$; $l_3$: $2/b+iT$ to $2/b+iU$; $l_4$: $\beta+iT$ to $2/b+iT$; $l_5$: $\beta-iT$ to $\beta+iT$; $l_6$: $2/b-iT$ to $\beta-iT$; $l_7$: $2/b-iU$ to $2/b-iT$; $l_8$: $\alpha-iU$ to $2/b-iU$; $l_9: \alpha-i\infty$ to $\alpha-iU$. 
By Cauchy's residue theorem,
\begin{align}\label{Cauchy}
    \frac{1}{2\pi i}\int_{\alpha-i\infty}^{\alpha+i\infty}{\frac{x^s\zeta(bs-1)\zeta(s-1)}{s(s+1)\zeta(bs)}}ds= \frac{x^2\zeta(2b-1)}{6\zeta(2b)}+\frac{b^2x^{2/b}\zeta(2/b-1)}{2(2+b)\zeta(2)}+\sum_{i=1}^9J_i.
\end{align}
The first two terms in the right side of \eqref{Cauchy} correspond to the residues at the poles of the integrand on the left side of \eqref{Cauchy} inside the modified contour of integration. Here $J_i$ is the integral of the integrand above along the line segments $l_i$. We now estimate $J_is$ using standard bounds for $\zeta(s)$ and $\displaystyle{\frac{1}{\zeta(s)}}$ (\cite [page 47]{Titch}),
\begin{displaymath}
 \zeta(\sigma+it) = \left\{
   \begin{array}{lr}
    \BigO{t^{\frac{1}{2}-\sigma} \log t},& \  -1\le \sigma \le 0,\\
     \BigO{t^{\frac{1- \sigma}{2}} \log t}, &\  0\le \sigma \le 1,\\
      \BigO{ \log t},  &\ 1\le\sigma \le 2,\\
       \BigO{1} , &\ \sigma \ge 2,
     \end{array}
    \right.
\end{displaymath}
and
\begin{displaymath}
 \frac{1}{\zeta(\sigma+it)} = \left\{
   \begin{array}{lr}
     \BigO{ \log t}, & \ 1\le\sigma \le 2,\\
      \BigO{1} , & \sigma \ge 2.
     \end{array}
    \right.
\end{displaymath}
In addition, we also use the Vinogradov-Koroborov zero free region 
(\cite{Vina}, \cite{Koro}),
\[ 
\sigma\ge 1-B{(\log t)}^{-\frac{2}{3}}{(\log\log t)}^{-\frac{1}{3}}, 
\]
in this region
\[\frac{1}{\zeta(s)}=\BigO{{(\log t)}^{\frac{2}{3}}{(\log \log t)}^{\frac{1}{3}}},\]
and $B$ is an absolute constant. 
Along the line segments $l_1$ and $l_8$,  we have %$|\zeta(b\alpha-1+ibt)|=\BigO{1}$, $|\zeta(\alpha-1+it)|=\BigO{\log U}$ and
%$1/|\zeta(b\alpha+ibt)|=\BigO{1}$, so that
\begin{align*}
    |J_1|,|J_9|&=\BigO{\int_{U}^{\infty} \frac{|x^{\alpha+it}||\zeta(b\alpha-1+ibt)||\zeta(\alpha+1+it)|}{|\alpha+it||\alpha+1+it||\zeta(b\sigma+ibt)|}dt}\\
   & =\BigO{x^2\log U\int_{U}^{\infty}\frac{dt}{t^2}}
    =\BigO{\frac{x^2\log U}{U}}.
\end{align*}
Along the line segments $l_2$ and $l_8$, we estimate the integrals $J_2$ and $J_8$ separately for $b=2$, and $b\ge 3$. For $b=2$, we decompose the integrals into three parts in the range $1\le \sigma\le 3/2$, $3/2\le \sigma\le 2$, and $2\le \sigma\le \alpha$, so that
\begin{align*}
    |J_2|,|J_8|&=\text{O}\left(\int_{1}^{3/2} \frac{|x^{\sigma+iU}||\zeta(b\sigma-1+ibU)||\zeta(\sigma-1+iU)|}{|\sigma+iU||\sigma+1+iU||\zeta(b\sigma+ibU)|}d\sigma \right. \nonumber\\
&+\int_{3/2}^2 \frac{|x^{\sigma+iU}||\zeta(b\sigma-1+ibU)||\zeta(\sigma-1+iU)|}{|\sigma+iU||\sigma+1+iU||\zeta(b\sigma+ibU)|}d\sigma  \nonumber\\
&\left. +\int_{2}^\alpha \frac{|x^{\sigma+iU}||\zeta(b\sigma-1+ibU)||\zeta(\sigma-1+iU)|}{|\sigma+iTU||\sigma+1+iU||\zeta(b\sigma+ibU)|}d\sigma\right) \nonumber\\
&=\BigO{\int_{1}^{3/2} \frac{x^{\sigma}U^{1-\sigma/2}\log^2 U}{U^2}d\sigma + \int_{3/2}^2 \frac{x^{\sigma}U^{1-\sigma/2}\log U}{U^2}d\sigma + \int_{2}^\alpha \frac{x^{\sigma}\log U}{U^2}d\sigma}\nonumber\\
&=\BigO{\frac{\log^2 U}{U}\int_{1}^{3/2} \left(\frac{x}{\sqrt U}\right)^\sigma d\sigma +\frac{(\log U)}{U}\int_{3/2}^{2} \left(\frac{x}{\sqrt U}\right)^\sigma d\sigma+\frac{\log U}{U^2} \int_{2}^\alpha {x^{\sigma}}d\sigma}\nonumber\\
&=\BigO{\frac{x^2\log U}{U^2}}. \label{J2J8a}
\end{align*}
We now estimate $J_2,$ and $J_8$ for $b\ge3$, for which we decompose the limits of the integral into $2/b\le\sigma<1$, $1\le\sigma\le2$, and $2\le\sigma\le\alpha$, and obtain
\begin{align*} 
    |J_2|,|J_8|&=\text{O}\left(\int_{2/b}^{1} \frac{|x^{\sigma+iU}||\zeta(b\sigma-1+ibU)||\zeta(\sigma-1+iU)|}{|\sigma+iU||\sigma+1+iU||\zeta(b\sigma+ibU)|}d\sigma \right. \nonumber\\
&+\int_{1}^2 \frac{|x^{\sigma+iU}||\zeta(b\sigma-1+ibU)||\zeta(\sigma-1+iU)|}{|\sigma+iU||\sigma+1+iU||\zeta(b\sigma+ibU)|}d\sigma  \nonumber\\
 & \left.+\int_{2}^\alpha \frac{|x^{\sigma+iU}||\zeta(b\sigma-1+ibU)||\zeta(\sigma-1+iU)|}{|\sigma+iTU||\sigma+1+iU||\zeta(b\sigma+ibU)|}d\sigma\right)\nonumber\\
&=\text{O}\left(\int_{2/b}^{1} \frac{x^{\sigma}U^{3/2-\sigma}\log^2 U}{U^2}d\sigma + \int_{1}^2 \frac{x^{\sigma}U^{1/2-\sigma}\log U}{U^2}d\sigma + \int_{2}^\alpha \frac{x^{\sigma}\log U}{U^2}d\sigma\right) \nonumber\\
&=\text{O}\left(\frac{\log^2 U}{\sqrt{U}}\int_{2/b}^{1} \left(\frac{x}{ U}\right)^\sigma d\sigma +\frac{(\log U)}{U^{3/2}}\int_{1}^{2} \left(\frac{x}{U}\right)^\sigma d\sigma+\frac{\log U}{U^2} \int_{2}^\alpha {x^{\sigma}}d\sigma\right) \nonumber\\
&=\text{O}\left(\frac{x^2\log U}{U^2}\right).
\end{align*}
On the line segments $l_3$ and $l_7$, $s=2/b+it, T\leq |t|\leq U$, we have 
\begin{align*}
    |J_3|,|J_7|&=\text{O}\left(\int_{T}^{U} \frac{|x^{2/b+it}||\zeta(1+ibt)||\zeta(2/b-1+it)|}{|1+it||2+it||\zeta(2+ibt)|}dt\right)
   =\text{O}\left(x^{2/b}\int_{T}^{U}\frac{t^{3/2-2/b}\log^2 t}{t^2}dt\right)\\
  & =\text{O}\left(x^{2/b}\int_{T}^{U}\frac{\log^2t}{t^{1/2+2/b}}dt\right)
    =\text{O}\left(x^{2/b}U^{1/2-2/b}\log^2U \right)+\text{O}\left(x^{2/b}T^{1/2-2/b}\log^2T\right).
\end{align*}
For $b=2, 3$, since $T<U$, the above error can be estimated to be
    \[|J_3|,|J_7|=\text{O}\left(x^{2/b}T^{1/2-2/b}\log^2T\right).\]
And, for $b\ge 4$, we have
    \[|J_3|,|J_7|=\text{O}\left(x^{2/b}U^{1/2-2/b}\log^2U \right).\]
Next, on the line segments $l_4$, and $l_6$, for $b=2$, we have
\begin{align*}
    |J_4|,|J_6|&= \text{O}\left(\int_{\beta}^{1} \frac{|x^{\sigma+iT}||\zeta(2\sigma-1+2iT)||\zeta(\sigma-1+iT)|}{|\sigma+iT||\sigma+1+iT||\zeta(2\sigma+2iT)|}d\sigma\right) \nonumber \\
&=\text{O}\left(\int_{\beta}^{1} \frac{x^{\sigma} T^{5/2-2\sigma}\log^3 T}{T^2}d\sigma\right)= \text{O}\left(T^{1/2}\log^3T\int_{\beta}^{1}\left(\frac{x}{T^2}\right)^\sigma d\sigma\right) \nonumber\\
&=\text{O}\left(\frac{x\log^3T}{T^{3/2}}\right). 
\end{align*}
For $b\ge3$, we break the integral from $\beta$ to 1 into $\beta$ to $1/b$, and $1/b$ to $2/b$, and obtain the following:
\begin{align*}
    |J_4|,|J_6|&= \BigO{\int_{\beta}^{2/b} \frac{|x^{\sigma+iT}||\zeta(b\sigma-1+ibT)||\zeta(\sigma-1+iT)|}{|\sigma+iT||\sigma+1+iT||\zeta(b\sigma+ibT)|}d\sigma} \nonumber\\
    &=\BigO{\int_{\beta}^{1/b} \frac{x^{\sigma}T^{3-(b+1)\sigma}{(\log T)}^{8/3}(\log\log T)^{1/3}}{T^2}d\sigma}\\&+ \BigO{\int_{1/b}^{2/b} \frac{x^{\sigma}T^{5/2-(b+2)\sigma/2}\log^3 T}{T^2}d\sigma} \nonumber \\
&=\text{O}\left(\frac{x^{1/b}{(\log T)}^{8/3}(\log\log T)^{1/3}}{T^{1/b}}\right)+\text{O}\left(\frac{x^{2/b}\log^3 T}{T^{1/2+2/b}}\right)=\text{O}\left(\frac{x^{2/b}\log^3 T}{T^{1/2+2/b}}\right). \label{J4J6b}
\end{align*}
Lastly, along the line segment $l_5$, for all $b\ge 2$, we have
\begin{align*}
    |J_5|&=\text{O}\left(\int_{-T}^{T} \frac{|x^{\beta+it}||\zeta(b\beta-1+ibt)||\zeta(\beta-1+it)|}{|\beta+it||\beta+1+it||\zeta(b\beta+ibt)|}dt\right)\\
   & =\text{O}\left(x^\beta \int_{-T}^{T}\frac{t^{3-(b+1)\beta}(\log(2+|t|) )^{8/3}(\log\log(3+|t|)^{1/3})}{1+t^2}dt\right)
   =\text{O}\left(x^\beta\right).
\end{align*}
Collecting all the above estimates for $b=2, 3$, and setting $U=x^{3/2}$,
and $\displaystyle{T=\exp\left(\frac{c_1{(\log x)}^{3/5}}{{(\log \log x)}^{1/5}}\right)}$, one obtains Theorem \ref{TheoremOne}. Likewise, for $b\ge 4$, setting $U=x$, and $T$ as above, one obtains Theorem \ref{TheoremTwo}.
\subsection{Mean value estimate}
In this section, we derive an asymptotic formula for the average of the generalized gcd-sum function. For $b=1$, this has been done in \cite[Theorem 4.7]{Brou}.
\begin{proof}[Proof of Theorem \ref{TheoremThree}]
\begin{align*}
       \sum_{n\le x} h_b(n)  & = \sum_{n\le x} \sum_{j\leq n}(j,n)_b
   =\sum_{n\le x}\sum_{d^b|n}d\sum_{\substack{j\leq n \\ (j,n)_b=d}}1 =\sum_{n\le x}\sum_{d^b|n}d\sum_{\substack{j\leq n \\ (\frac{j}{d^b},\frac{n}{d^b})_b=1}}1\\
   &=\sum_{n\le x}\sum_{d^b|n}d\phi_b(\frac{n}{d^b})=\sum_{d \le x^{1/b}}d \sum_{k\le \frac{x}{d^b}}\phi_b(k)
         \end{align*}   
          McCarthy \cite{McCarthy} computed the average order of Klee's totient function $\phi_b(n)$ to be
  \[ \sum_{n\le x}\phi_b(n)= \frac{x^2}{2\zeta(2b)}+ \text{O}(x).\]
    which implies
    \begin{align*}
        \sum_{n\le x} h_b(n)  & =\sum_{d \le x^{1/b}}d\left[\frac{x^2}{2\zeta(2b)d^{2b}}+\text{O}\left(\frac{x}{d^b}\right)\right]\\
        & =\frac{x^2}{2\zeta(2b)}\sum_{d \le x^{1/b}}\frac{1}{d^{2b-1}}+\text{O}\left(x\sum_{d \le x^{1/b}}\frac{1}{d^{b-1}}\right)
    \end{align*}.
Therefore we get the required result,
    \[\sum_{n\le x} h_b(n)=\frac{\zeta(2b-1)}{2\zeta(2b)}x^2+E_b(x),\numberthis\label{hn2}\]
    where $E_2(x)=\BigO{x\log x}$, and $E_b(x)=\BigO{x}$ for $b\ge 3.$
    \end{proof}
    \section{Harmonic Mean}
    The harmonic mean $H_b(n)$ defined in \eqref{harmonicmean} is a multiplicative function and has a Dirichlet series representation. The following lemmas cover these properties, using which we prove Theorem \ref{TheoremFive}.
    \begin{lem}\label{Dirichharm}
    The function $H_b(n)$ is multiplicative and its value at prime powers is given by,
    \begin{displaymath}
 H_b(p^{kb+j}) = \left\{
   \begin{array}{lr}
    1,& \ 0\le j\le b-1, \ k=0,\\
     \dfrac{p^{k(b+1)}(p^{b+1}-1)}{p^{k(b+1)}(p^{b+1}-p)+p-1}, &\  0\le j\le b-1, \ k\in\mathbb{N}.\\
     \end{array}
    \right.
\end{displaymath}
    \end{lem}
    \begin{proof}
    From \eqref{harmonicmean}, we have 
    \begin{align*}\label{hb2}
        \frac{n}{H_b(n)} & = \sum_{k=1}^n\frac{1}{(k,n)_b}
   =\sum_{d^b|n}\frac{1}{d}\sum_{\substack{j\leq n \\ (j,n)_b=d}}1 %=\sum_{d^b|n}\frac{1}{d}\sum_{\substack{j\leq n \\ \gcd_b(\frac{j}{d^b},\frac{n}{d^b})=1}}1
   =\sum_{d^b|n}\frac{1}{d}\phi_b(\frac{n}{d^b})
  %=\sum_{m|n}\frac{1}{{m}^{1/b}}\phi_b(\frac{n}{m})\chi_b(m)
  =n^{-1/b}\sum_{m|n}m^{1/b}\chi_b\left(\frac{n}{m}\right)\phi_b(m),\numberthis  
         \end{align*}
         where $\chi_b$ and $\phi_b$ as in \eqref{chib}, and \eqref{phib} respectively. 
     The sum in the right side of \eqref{hb2} is a convolution of multiplicative functions, and so $H_b(n)$ is multiplicative with 
         \[H_b(p^n)=1 \ ; \ 0\le n\le b-1. \]
         For $n=p^{kb+j}$, where $0\leq j\le b-1$ and $k\in\mathbb{N}$, we have
\begin{align*}
        \frac{H_b(p^{kb+j})}{p^{kb+j}}&=p^{k+j/b}\left(\sum_{d|p^{kb+j}}d^{1/b}\chi_b\left(\frac{n}{d}\right)\phi_b(d)\right)^{-1}=p^{k+j/b}\left(p^{j+j/b}+\sum_{l=1}^kp^{l+j/b}\phi_b(p^{lb+j})\right)^{-1}\\
        & = p^k\left(p^j+\sum_{l=1}^klp^{l+lb+j-b}(p^b-1)\right)^{-1}= p^{k-j}\left(1+\frac{(p^b-1)p( p^{k(b+1)}-1)}{(p^{b+1}-1)}\right)^{-1}\\
        & = \frac{p^{k-j}(p^{b+1}-1)}{p^{k(b+1)}(p^{b+1}-p)+p-1}.
        \end{align*}
        \end{proof}
        \begin{lem} \label{lemharmonic}
        The Dirichlet series associated to the harmonic mean $H_b(n)/n$ has an Euler product representation and converges absolutely for $\Re(s)>0$, and can be represented as
        \[\sum_{n=1}^\infty\frac{H_b(n)/n}{n^s}=\frac{\zeta(s+1)}{\zeta(bs+b)}\prod_p\left(1+K_b(s)\right), \numberthis\label{Dirichletharmonic}\]
        where 
        \[K_b(s):=(p^{b+1}-1)\sum_{k=1}^\infty\frac{p^{k}}{p^{k(bs-1)}(p^{k(b+1)}(p^{b+1}-p)+p-1)}.\]
        \end{lem}
        \begin{proof}
Using Lemma \ref{Dirichharm}, the Dirichlet series associated to $H_b(n)$ can be written as
\begin{align*}
       \sum_{n=1}^\infty\frac{H_b(n)/n}{n^s} & = \prod_p\left(1+\sum_{m=1}^\infty \frac{H_b(p^m)/p^m}{p^{ms}}\right) \\&=\prod_p \left(\sum_{j=0}^{b-1}\frac{1/p^j}{p^{js}}+\sum_{j=0}^{b-1}\sum_{k=1}^{\infty}\frac{p^{k-j}(p^{b+1}-1)}{p^{(kb+j)s}(p^{k(b+1)}(p^{b+1}-p)+p-1)}\right)\\
       & = \prod_p\left(\frac{1-1/p^{b(s+1)}}{1-1/p^{s+1}}+\frac{1-1/p^{b(s+1)}}{1-1/p^{s+1}}\sum_{k=1}^\infty\frac{p^{k}(p^{b+1}-1)}{p^{kbs}(p^{k(b+1)}(p^{b+1}-p)+p-1)}\right)\\
       & = \frac{\zeta(s+1)}{\zeta(bs+b)}\prod_p\left(1+(p^{b+1}-1)\sum_{k=1}^\infty\frac{1}{p^{k(bs-1)}(p^{k(b+1)}(p^{b+1}-p)+p-1)}\right)\\
       & =\frac{\zeta(s+1)}{\zeta(bs+b)}\prod_p(1+K_b(s))=\frac{\zeta(s+1)}{\zeta(bs+b)}A_b(s),\numberthis \label{harm}
\end{align*}

where
$A_b(s):=\prod_p(1+K_b(s))$, and 
\[K_b(s)=(p^{b+1}-1)\sum_{k=1}^\infty\frac{1}{p^{k(bs-1)}(p^{k(b+1)}(p^{b+1}-p)+p-1)}.\] 
  .
  %=\prod_p\left(1+(p^{b+1}-1)\sum_{k=1}^\infty\frac{1}{p^{k(bs-1)}(p^{k(b+1)}(p^{b+1}-p)+p-1)}\right)
Let $s=\sigma+it$. Then
\begin{align*}
    \left|\dfrac{1}{p^{k(bs-1)}(p^{k(b+1)}(p^{b+1}-p)+p-1)}\right|&\le
    \left|\dfrac{1}{p^{k(b\sigma-1)}(p^{k(b+1)}(p^{b+1}-p)+p-1)}\right|\\&\le \dfrac{1}{p^{k(b\sigma-1)}p^{k(b+1)}(p^{b+1}-p)}, 
\end{align*}
and so
%\left|(p^{b+1}-1)\sum_{k=1}^\infty\frac{1}{p^{k(bs-1)}(p^{k(b+1)}(p^{b+1}-p)+p-1)}\right|\\
 %&\leq (p^{b+1}-1)\sum_{k=1}^\infty\frac{1}{|p^{k(bs-1)}|(p^{k(b+1)}(p^{b+1}-p)+p-1)}\\
  %&= (p^{b+1}-1)\sum_{k=1}^\infty\frac{1}{p^{k(b\sigma-1)}(p^{k(b+1)}(p^{b+1}-p)+p-1)}\\ 
  % &\leq (p^{b+1}-1)\sum_{k=1}^\infty\frac{1}{p^{k(b\sigma-1)}p^{k(b+1)}(p^{b+1}-p)}\\
  \begin{align*}
 |K_b(s)| &= \frac{(p^{b+1}-1)}{(p^{b+1}-p)}\sum_{k=1}^\infty\frac{1}{p^{k(b\sigma-1)}p^{k(b+1)}}\\
    &=\frac{(p^{b+1}-1)}{(p^{b+1}-p)}\sum_{k=1}^\infty\frac{1}{p^{k(b\sigma+b)}}=\frac{(p^{b+1}-1)}{(p^{b+1}-p)}\frac{1/p^{(b\sigma+b)}}{1-1/p^{(b\sigma+b)}}\\
    &=\frac{(1-p^{-b-1})}{(1-p^{-b})}\frac{1}{p^{(b\sigma+b)}-1}.\\
    \end{align*}
        Since $\dfrac{1}{1-p^{-b}}\leq 2$, and $1-p^{-b-1}\leq 1$, we have
        \[\sum_{p}|K_b(s)|\leq \sum_p\frac{2}{p^{b\sigma+b}-1}<\infty.\]
        Hence, the product $A_b(s)=\prod_p(1+K_b(s))$ is absolutely convergent for $\sigma_o>-1+1/b$, and analytic on the half plane $\Re s>-1+1/b$. And for any fixed $\sigma_o>-1+1/b$, we have
        \begin{align*}
            |A_b(s)| &\leq \prod_p(1+|K_b(s)|)\leq \prod_p\left(1+\frac{2}{p^{b\sigma+b}-1}\right)\\
            &= \prod_p\frac{1+p^{b\sigma_o+b}}{1-p^{b\sigma_o+b}}=\frac{\zeta^2(b\sigma_o+b)}{\zeta(2b\sigma_o+2b)}<\infty.
        \end{align*}
        Thus, $A_b(s)$ is uniformly bounded for $\Re s>-1+1/b$. It follows that $A_b(s)$ is absolutely convergent and defines an analytic function on the half plane $\Re s>-1+1/b$. This shows that \eqref{Dirichletharmonic} has a meromorphic continuation to the half plane $\Re(s)>-1+1/b.$ 
\end{proof}
We now turn to proving Theorem \ref{TheoremFive}.
\begin{proof}[Proof of Theorem \ref{TheoremFive}]
Expanding $A_b(s)$ into a Dirichlet series
\[A_b(s)=\sum_{n=1}^\infty\frac{a_b(n)}{n^{s}},\] which is absolutely convergent for $\Re s>-1+1/b$, 
Now, from \cite[D-24]{Series} we have, $\dfrac{\zeta(s)}{\zeta(bs)}= \mathlarger{\sum}_{n=1}^\infty\dfrac{\xi_b(n)}{n^{s}}$ where $\xi_b(n)$ is the characteristic function for $b$-free numbers.
 %\begin{displaymath}
 %\xi_b(n) = \left\{
  % \begin{array}{lr}
   % 1; & \text{if n is b free},\\
    % 0; & \text{otherwise}.\\
     %\end{array}
    %\right.
%\end{displaymath} 
It follows that \[\dfrac{\zeta(s+1)}{\zeta(bs+b)}= \sum_{n=1}^\infty\dfrac{\xi_b(n)/n}{n^{s}}\numberthis\label{xiseries}.\]
 By Lemma \ref{lemharmonic} and \eqref{xiseries}, we obtain
\[\sum_{n=1}^\infty\frac{H_b(n)/n}{n^s}=\sum_{n=1}^\infty\frac{\xi_b(n)/n}{n^{s}}\sum_{n=1}^\infty\frac{a_b(n)}{n^{s}}.\]
This gives
\[\frac{H_b(n)}{n}=\sum_{d|n}\frac{\xi_b(d)}{d}a_b(\frac{n}{d}).\numberthis\label{meanformula}\]
Since the Dirichlet series $\sum_{n=1}^\infty a_b(n)n^{-s}$ is absolute convergent for $\Re s>-1+1/b$. Hence for any $ U>1,$ 
\[\sum_{U<n\le 2U}|a_b(n)| \ll U^{-1+1/b+\epsilon}.\]
Thus, the infinite series $\sum_{m\ge 1}g_b(m)$ is convergent and
\[
\sum_{m\leq x}a_b(m)=C+\text{O}(x^{-1+1/b+\epsilon}),\]
where $C=A_b(0)$ is constant. Using \eqref{meanformula} and the above expression, we obtain
\begin{align*}
    \sum_{n\leq x}\frac{H_b(n)}{n}&= \sum_{m\leq x}\frac{\xi_b(m)}{m}\sum_{n\leq x/m}a_b(n)= \sum_{m\leq x}\frac{\xi_b(m)}{m}\left(C+\text{O}\left(\frac{x^{-1+1/b+\epsilon}}{m^{-1+1/b+\epsilon}}\right)\right)\\
    &=C \sum_{m\leq x}\frac{\xi_b(m)}{m}+\text{O} \left(x^{-1+1/b+\epsilon}\sum_{m\leq x}\frac{\xi_b(m)}{m^{1/b+\epsilon}}\right)\numberthis\label{haverage}
\end{align*}
    The average order of $\xi_b(n)$ is estimated by Walfisz in \cite[Page192]{Walf} as
    \begin{align*}
        \sum_{n\leq x}\xi_b(n)&= \frac{x}{\zeta(b)}+\text{O}(x^{1/b}\exp\{\log^{3/5}x(\log\log x)^{-1/5}\})\\
    &=\frac{x}{\zeta(b)}+\text{O}(E_x),
    \end{align*}
    where $ E_x=x^{1/b}\exp\{\log^{3/5}x(\log\log x)^{-1/5}\}$. Using this estimate and applying summation by parts, we have
    \begin{align*}
        \sum_{m\leq x}\frac{\xi_b(m)}{m}
        %\left[\frac{x}{\zeta(b)}+\text{O}(E_x)\right]\frac{1}{x}+\int_1^x\left[\frac{t}{\zeta(b)}+\text{O}(E_ t)\right]\frac{1}{t^2}dt\numberthis\label{xi/m}\\
        &= \frac{1}{\zeta(b)}+\frac{\log x}{\zeta(b)}+\text{O}\left(\frac{E_x}{x}\right), 
        \end{align*} and
\[\sum_{m\leq x}\frac{\xi_b(m)}{m^{1/b+\epsilon}} %\left[\frac{x}{\zeta(b)}+\text{O}(E_x)\right]\frac{1}{x^{1/b+\epsilon}}+\int_1^x\left[\frac{t}{\zeta(b)}+\text{O}(E_ t)\right]\frac{1}{t^{1+1/b+\epsilon}}dt \\ 
=\BigO{x^{1-1/b-\epsilon}}.\]
    Collecting the above results, and substituting in \eqref{haverage}, we get the required result in theorem \ref{TheoremFive}. 
  \end{proof}
  \section{Geometric mean}
  In this section, we discuss the average value of geometric mean $G_b(n),$ and give a proof of Theorem \ref{TheoremSix}. The method of proof is similar to the proof for harmonic means. We first define the relevent functions. Recall from \eqref{Geometricmean},
 
  %G_b(n)&=\left(\prod_{j\le n}{\gcd}_b(j,n)\right)^{1/n}\\
     \[(G_b(n))^n=\prod_{j\le n}(j,n)_b=\left(\prod_{d^b|n}d^{\phi_b\left(\frac{n}{d^b}\right)}\right).\]
     This implies
      \begin{align*} 
     n\log G_b(n)&=\sum_{d^b|n}\phi_b\left(\frac{n}{d^b}\right)\log d=\frac{1}{b}\sum_{m|n}\phi_b\left(\frac{n}{m}\right)\chi_b(m)\log m.\numberthis\label{gbn}
  \end{align*}
  Here $\chi_b$ and $\phi_b$ is same as in \eqref{chib}, and \eqref{phib} respectively. It follows that the function $n\log G_b(n)$ is a convolution of functions $\phi_b(n)$, and $\chi_b(n)\log n$. So, the Dirichlet series associated to $n \log G_b(n)$ is the product of Dirichlet series associated to  $\phi_b(n)$ and $\chi_b(n)\log n$.  Now, \[\sum_{n=1}^\infty \frac{\chi_b(n)}{n^s}=\zeta(bs) \ \ \text{for} \ \ \Re(s)>1/b,\numberthis\label{Dchib}\]
  and  \[\sum_{n=1}^\infty \frac{\chi_b(n){\log}n}{n^s}=-b\zeta'(bs) \ \ \text{for} \ \ \Re(s)>1/b\numberthis\label{Dchiblog}\]
  Using \eqref{Dirichlet-phi} and \eqref{Dchiblog}, we conclude that $n \log G_b(n)$ has an absolutely convergent Dirichlet series for $\Re s>2$ given by 
  \[\sum_{n=1}^\infty \frac{n \log G_b(n)}{n^s}=-\frac{\zeta(s-1)\zeta'(bs)}{\zeta(bs)}.\] We are now equipped to prove Theorem \ref{TheoremSix}. 
  \begin{proof}[Proof of Theorem \ref{TheoremSix}]
  From \eqref{gbn}, we have
  \begin{align*}
     \sum_{n\le x}n \log G_b(n) =\frac{1}{b}\sum_{n\le x}\sum_{m|n}\phi_b\left(\frac{n}{m}\right)\chi_b(m){\log}m= \frac{1}{b}\sum_{m\le x}\chi_b(m){\log}m \sum_{k\le x/m}\phi_b(k).
   \end{align*}
   McCarthy \cite{McCarthy} computed the average order of Klee's totient function $\phi_b(n)$ to be
  \[ \sum_{n\le x}\phi_b(n)= \frac{x^2}{2\zeta(2b)}+ \text{O}(x).\]
  This gives
   \begin{align*}
    \sum_{n\le x}n \log G_b(n) &= \frac{1}{b}\sum_{m\le x}
    \chi_b(m)\log m\left[\frac{1}{2\zeta(2b)}\frac{x^2}{m^2}+ \text{O}\left(\frac{x}{m}\right)\right] \\
    &= \frac{x^2}{2b\zeta(2b)}\sum_{m\le x}\frac{\chi_b(m)\log m}{m^2}+ \BigOb{x\sum_{m\le x}\frac{\chi_b(m)\log m}{m}}.\numberthis\label{averagegbn}
    \end{align*}
    Moreover, the Dirichlet series $\sum_{n=1}^\infty\frac{\chi_b(m)\log m }{m^s}$ is absolutely convergent for $\Re s>1/b$. Hence for every $\epsilon>0$, and $ U>1 $, 
 \[\sum_{U<n\le 2U}|\chi_b(n)\log (n)| \ll U^{1/b+\epsilon}.\] 
Therefore from \eqref{Dchib}, and \eqref{Dchiblog}, for every $\epsilon>0$,
%and it is also imply that the infinite series $\sum_{m>1}\dfrac{\chi_b(m)\log m}{m}$ and $\sum_{m>1}\dfrac{\chi_b(m)\log m}{m^2}$  are  convergent such that
\[\sum_{m\leq x}\frac{\chi_b(m)\log m}{m}=-b\zeta'(b)+\text{O}(x^{-1+1/b+\epsilon}),
\]
and
\[
  \sum_{m\leq x}\frac{\chi_b(m)\log m}{m^2}=-b\zeta'(2b)+\text{O}(x^{-2+1/b+\epsilon}).
\]
Collecting the above estimates, and substituting in \eqref{averagegbn}, we obtain Theorem \ref{TheoremSix}.
  \end{proof}
  \section{Asymptotics of $M_{f,b,k}(x)$}
    \subsection*{Proof of Theorem \ref{TheoremFour}}
    % The proof follows the work from Theorem 4 of \cite{Visibility} where the asymptotic of the average value of ${\gcd}_b$ is proved. Using similar methods as in that paper, one can obtain results for mean values for more general arithmetic functions. For example, as in Theorem 2 of \cite{Visibility}, one can show that if $f$ is an arithmetic function satisfying
    % \[\frac{1}{N}\sum_{k=1}^{N}\frac{|(f*\mu)(k)|}{k}\longrightarrow 0 \ \text{as} \  N\rightarrow \infty,\]
    % then the mean value of $f((m,n)_b)$ over $\mathbb{N}\times\mathbb{N}$ exists and is equal to 
     %$\dfrac{\zeta_f(2b)}{\zeta(2b)}$, 
%where $\zeta_f(2b)=\mathlarger{\sum}_{n=1}f(n)n^{-2b}$.
We separate the diagonal terms and use the symmetry of the gcd function to write the following identity
\[\sum_{m,n\leq x}(m,n)_b=2\sum_{n\le x}\sum_{m=1}^{n}(m,n)_b-\sum_{n\le x}(n,n)_b.\numberthis\label{7.1}\]
The diagonal sum in \eqref{7.1} can be estimated as
\begin{align*}
    \sum_{n\le x}(n,n)_b &= \sum_{d\le x^{1/b}} \sum_{d|(n,n)_b}\phi(d) = \sum_{d\le x} \sum_{d^b|n}\phi(d)= \sum_{d\le x^{1/b}}\phi_(d)\left(\frac{x}{d^b}+\text{O}(1)\right)\\
    & = x \sum_{d\le x^{1/b}}\frac{\phi_(d)}{d^b}+\BigO{\sum_{d\le x^{1/b}}\phi(d)}
\end{align*}
  We use the following well known estimates to estimate the above.
  \[\sum_{n\le x}\phi(n)=\frac{x^2}{2\zeta(2)}+\text{O}(x\log x),\]
  \[\sum_{n\leq x}\frac{\phi(n)}{n^2} = \frac{\log x}{\zeta(2)}+\frac{\gamma}{\zeta(2)}-A +\text{O}\left(\frac{\log x}{x}\right),\]
        where $A= \sum_{n=1}^{\infty}\frac{\mu(n)\log n}{n^2}=\frac{\zeta'(2)}{\zeta(2)^2}$, and for $\alpha \ge 1, \alpha \neq 2,$
               \[ \sum_{n\leq x}\frac{\phi(n)}{n^\alpha}= \frac{\zeta(\alpha-1)}{\zeta(\alpha)}+\frac{x^{2-\alpha}}{(2-\alpha)\zeta(2)}-A +\text{O}\left(x^{1-\alpha}\log x\right).\]
Using the above results, and the estimate from \eqref{hn2},  we have
%\[\sum_{n\le x}(n,n)_2=\frac{x\log x}{2\zeta(2)}+\text{O}(x)\numberthis\label{6.4}\]
%\[\sum_{n\le x}(n,n)_b=\frac{\zeta(b-1)}{b\zeta(b)}x+\text{O}(x^{2/b})\numberthis\label{6.5}\]
%collecting the estimates from \eqref{hn2}, and \eqref{6.4}, \eqref{6.5} and put in \eqref{7.1}, we have
\[\sum_{m,n\leq x}(m,n)_2=\frac{\zeta(3)}{\zeta(4)}x^2+\text{O}(x \log x),\] and 
\[\sum_{m,n\leq x}(m,n)_b=\frac{\zeta(2b-1)}{\zeta(2b)}x^2+\text{O}_b(x).\]
For the $k$-generalized gcd function, $k>2$, and for any arbitrary arithmetical function $f$, we can write
\[\sum_{n_1,n_2,\cdots,n_k\le x}f(n_1,n_2,\cdots,n_k)_b =\sum_{n_1,n_2,\cdots,n_k\le x}\sum_{d|(n_1,n_2,\cdots,n_k)_b}r(d),\]
where $r(n)$ is an arithmetical function such that $f(n)=1\ast r(n)$. Then by Mobius inversion formula, we have the identity
%\[r(n)=(f*\mu)(n)\]
%which implies 
\begin{align*}
    \sum_{n_1,n_2,\cdots,n_k\le x}f(n_1,n_2,\cdots,n_k)_b &=\sum_{d\le x^{1/b}}(f*\mu)(d){\left\lfloor{\frac{x}{d^b}}\right\rfloor}^k. 
\end{align*}
For  $b\ge 2$, $k > 2$ and $f= id$, the above identity leads to the following asymptotic formula
\[\sum_{n_1,n_2,\cdots,n_k\le x}(n_1,n_2,\cdots,n_k)_b=\frac{\zeta(bk-1)}{\zeta(bk)}x^k+\BigOb{x^{k-1}}.\]
\section{Generalized lcm}
For any multiplicative, arithmetical function $f\in \mathcal{A}_r$ we will estimate the average order of $f([n_1,n_2,...,n_k]_b)$. We use the method of \cite{Hilber} where the authors do it for $b=1$. We will only briefly sketch parts of the proof which have been modified from \cite{Hilber}. The generalized lcm is a multiplicative function of $k$ variables, and so $f([n_1,n_2,...,n_k]_b)$ is multiplicative with Dirichlet series described in the following lemma. 
\begin{lem}\label{lcmlem}
 For $b\ge 2$, $k\ge 2$ and  $f\in \mathcal{A}_r$ with $r>-1$,
\[L^b_{f,k}(z_1,\cdots, z_k) =\sum_{n_1,\cdots,n_k=1}^\infty \frac{f([n_1,\cdots, n_k]_b)}{n_1^{z_1},\cdots, n_k^{z_k}}
= \zeta(z_1 - r)\cdots\zeta(z_k - r)H^b_{f,k}(z_1,\cdots , z_k),\]
where $H^b_{f,k}(z_1,\cdots, z_k)$ is absolutely convergent for
 \[ \Re z_1,\cdots,\Re z_k > A := \left\{\begin{array}{ll}
         r+\frac{1}{2} & \mbox{if } r \ge 0 \\
         \frac{r+1}{2} & \mbox{if } -1/2 \le r\le 0.
         \end{array}\right.\]
\end{lem}
\begin{proof}
Since $f$ is a multiplicative function, 
\[L^b_{f,k}(z_1,\cdots, z_k) =\sum_{n_1,\cdots,n_k=0}^\infty \frac{f([n_1,\cdots, n_k]_b)}{n^{z_1},\cdots, n^{z_k}}= \prod_p\sum_{v_1,\cdots,v_k=0}^\infty \frac{f(p^{\lceil\frac{\max(v_1,\cdots,v_k)}{b}\rceil})}{p^{v_1z_1+\cdots+v_kz_k}}\numberthis\label{lcm1}.\]
As in \cite{Hilber}, we consider the cases $r\ge 0$, and $-1/2\le r\le 0$ separately.
\begin{case}
 Consider $r\ge 0$ and $\Re z_1,\cdots, \Re z_k \ge \delta \ge r$. Divide the sum of \eqref{lcm1} into two parts 
\[L^b_{f,k}(z_1, \cdots, z_k)=\prod_p\left(1+\frac{f(p)}{p^{z_1}}+\cdots +\frac{f(p)}{p^{z_k}}+\sum_{v_1+\cdots+v_k\ge 2}\dfrac{f(p^{\lceil\frac{\max(v_1,\cdots,v_k)}{b}\rceil})}{p^{v_1z_1+\cdots+v_kz_k}}\right).\numberthis\label{lcm3}\]
Using properties of a function $f\in\mathcal{A}_r$, and the fact $r\ge 0$, we have 
\[\frac{f(p)}{p^{z_i}}= \frac{1}{p^{z_i-r}}+\text{O}\left(\frac{1}{p^{\delta-r+1/2}}\right).\]
Moreover, \[\left|\frac{f(p^{\lceil\frac{\max(v_1,\cdots,v_k)}{b}\rceil})}{p^{v_1z_1+\cdots+v_kz_k}}\right|\le C_3\frac{p^{\frac{\max(v_1,\cdots,v_k)}{b}+1}}{p^{\delta(v_1+\cdots+v_k)}}=\text{O}\left(\frac{p^r}{p^{2(\delta-r/b)}}\right).\]
This gives,
\begin{align*}
    &L^b_{f,k}(z_1,\cdots, z_k){\zeta(z_1-r)}^{-1}\cdots{\zeta(z_k-r)}^{-1}\\&=\prod_p\left(1-\frac{1}{p^{z_1-r}}\right)\cdots\left(1-\frac{1}{p^{z_k-r}}\right)\left[1+\frac{1}{p^{z_1-r}}+\cdots+\frac{1}{p^{z_k-r}}+\text{O}\left(\frac{1}{p^{\delta-r+1/2}}\right)+\text{O}\left(\frac{1}{p^{2(\delta-r/b)-r}}\right)\right]\\
    &=\prod_p\left(1+\text{O}\left(\frac{1}{p^{\delta-r+1/2}}\right)+\text{O}\left(\frac{1}{p^{2(\delta-r)}}\right)+\text{O}\left(\frac{1}{p^{2(\delta-r/b)-r}}\right)\right)
\end{align*}
 The above infinite product is absolutely convergent for $\delta>r+\frac{1}{2}$.
 \end{case}
\begin{case}
  Let $-1/2\le r\le 0$ and $\Re z_1,\cdots, \Re z_k \ge \delta \ge 0$. Dividing the sum in \eqref{lcm1} into two parts, we have
\[L^b_{f,k}(z_1,\cdots , z_k)=\prod_p\left(1+\sum_{\max(v_1,,v_k)=1}\frac{f(p)}{p^{v_1z_1+\cdots+v_kz_k}}+\sum_{\max(v_1,\cdots,v_k)\ge 2}\frac{f(p^{\lceil\frac{\max(v_1,\cdots,v_k)}{b}\rceil})}{p^{v_1z_1+\cdots+v_kz_k}}\right)\numberthis\label{lcm2}\]
If $v_i=1$ for a unique index $i$, $1\le i\le k$ in the first sum of \eqref{lcm2}, one has
\[\frac{f(p)}{p^{z_i}}= \frac{1}{p^{z_i-r}}+\text{O}\left(\frac{1}{p^{\delta-r+1/2}}\right).\]
If at least two $v_i=1$, where $1\le i\le k$, then we have
\[\left|\frac{f(p)}{p^{v_1z_1+\cdots+v_kz_k}}\right|=\text{O}\left(\frac{1}{p^{2\delta-r}}\right).\]
The second sum of \eqref{lcm2} becomes
    \[\left|\frac{f(p^{\lceil\frac{\max(v_1,\cdots,v_k)}{b}\rceil})}{p^{v_1z_1+\cdots+v_kz_k}}\right|\le C_2\frac{p^{r(\frac{\max(v_1,\cdots,v_k)}{b}+1)}}{p^{\delta(v_1+\cdots+v_k)}}\le C_2 \frac{p^{r}}{p^{(\delta-r/b)\max(v_1,\cdots,v_k)}}\]
\[\left|\frac{f(p^{\lceil\frac{\max(v_1,\cdots,v_k)}{b}\rceil})}{p^{v_1z_1+\cdots+v_kz_k}}\right|=\text{O}\left(\frac{1}{p^{2(\delta-r/b)-r}}\right).\]
This yields,
\begin{align*}
    L^b_{f,k}(z_1,\cdots, z_k){\zeta(z_1-r)}^{-1}\cdots{\zeta(z_k-r)}^{-1}
    %\\&=\prod_p\left(1-\frac{1}{p^{z_1-r}}\right)\cdots\left(1-\frac{1}{p^{z_k-r}}\right)\left[1+\frac{1}{p^{z_1-r}}+\cdots+\frac{1}{p^{z_k-r}}+\text{O}\left(\frac{1}{p^{\delta-r+1/2}}\right)+\text{O}\left(\frac{1}{p^{2\delta-r}}\right)+\text{O}\left(\frac{1}{p^{2(\delta-r/b)-r}}\right)\right]\\
    &=\prod_p\left(1+\BigO{\frac{1}{p^{\delta-r+1/2}}}+\BigO{\frac{1}{p^{2\delta-r}}} \right. \\  &+\left.\BigO{\frac{1}{p^{2(\delta-r/b)-r}}}\right)
\end{align*}
Since $r<0$, it follows that the above series is absolutely convergent for $\delta>\frac{r+1}{2}$.
\end{case}
\end{proof}
\begin{proof}[Proof of Theorem \ref{TheoremEight}]
Let $h^b_{f,k}(n_1,\cdots, n_k)$ be the coefficients in the multi-variable Dirichlet series of $H^b_{f,k}(z_1,\cdots, z_k)$. 
%into its Dirichlet series
%\[H^b_{f,k}(z_1,\cdots, z_k)=\sum_{n_1,\cdots,n_k=1}^\infty\frac{h^b_{f,k}(n_1,\cdots, n_k)}{n_1^{z_1},\cdots , n_k^{z_k}}\]
Using Lemma \ref{lcmlem}, we can write
\[f([n_1,\cdots, n_k]_b)= \sum_{d_1|n_1,\cdots,d_k|n_k}h^b_{f,k}(d_1,\cdots,d_k)\left(\frac{n_1}{d_1}\right)^r\cdots\left(\frac{n_k}{d_k}\right)^r,\]
which implies
\begin{align*}
\sum_{n_1,\cdots,n_k\le x}f([n_1,\cdots , n_k]_b)&= \sum_{n_1,\cdots,n_k\le x}\sum_{d_1|n_1,\cdots,d_k|n_k}h^b_{f,k}(d_1,\cdots,d_k)\left(\frac{n_1}{d_1}\right)^r\cdots\left(\frac{n_k}{d_k}\right)^r
\\
&=\sum_{d_1,\cdots,d_k\le x}h^b_{f,k}(d_1,\cdots,d_k)\sum_{m_1\le x/d_1,\cdots,m_k\le x/d_k}m_1^r\cdots m_k^r\\
%& = \sum_{d_1,\cdots,d_k\le x}h^b_{f,k}(d_1,\cdots,d_k)\left(\frac{x^{r+1}}{d_1^{r+1}}+\text{O}(\frac{x^R}{(r+1)d_1^R})\right)\cdots\left(\frac{x^{r+1}}{(r+1)d_k^{r+1}}+\text{O}(\frac{x^R}{d_k^R})\right)\\
&= \frac{x^{k(r+1)}}{(r+1)^k}\sum_{d_1,\cdots,d_k\le x}\frac{h^b_{f,k}(d_1,\cdots,d_k)}{d_1^{r+1}\cdots d_k^{r+1}}+T_{k,r}(x),\numberthis\label{lcm4}
\end{align*}
where %$R=\max(r,0)$  and 
$T_{k,r}(x)= \text{O}(x^{k(r+1)-\frac{1}{2}\min(r+1,1)+\epsilon})$.
The sum on the right side of \eqref{lcm4} can be evaluated as
\begin{align*}
    \sum_{d_1,\cdots,d_k\le x}\frac{h^b_{f,k}(d_1,\cdots,d_k)}{d_1^{r+1}\cdots d_k^{r+1}}& = \sum_{d_1,\cdots,d_k=1}^\infty\frac{h^b_{f,k}(d_1,\cdots,d_k)}{d_1^{r+1}\cdots d_k^{r+1}}-\sum_{1\le s\le k}\sum_{\substack{d_1,\cdots,d_s > x\\d_{s+1},\cdots,d_k\le x}}\frac{h^b_{f,k}(d_1,\cdots,d_k)}{d_1^{r+1}\cdots d_k^{r+1}}\\
    &=H^b_{f,k}(r+1,\cdots,r+1)+\text{O}(x^{k(r+1)-\frac{1}{2}\min(r+1,1)+\epsilon}),
\end{align*} 
where
%The tail can be bounded by
%\[\sum_{\substack{d_1,\cdots,d_s > x\\d_{s+1},\cdots,d_k\le 5x}}\frac{h^b_{f,k}(d_1,\cdots,d_k)}{d_1^{r+1}\cdots d_k^{r+1}}=\text{O}(x^{k(r+1)-\frac{1}{2}\min(r+1,1)+\epsilon}).\]
%The infinite series is convergent by Lemma \eqref{lcmlem} and contributes to the main term. It is given by
\begin{align*}
 C^b_{f,k}:=H^b_{f,k}(r+1,\cdots,r+1)=\prod_p\left(1-\frac{1}{p}\right)^k\sum_{v_1,\cdots,v_k=0}^\infty \frac{f(p^{\lceil\frac{\max(v_1,\cdots,v_k)}{b}\rceil})}{p^{(r+1)(v_1+\cdots+v_k)}}.
 \end{align*}
\end{proof}
\begin{proof}[Proof of Corollary \ref{CorOne}]
Let $k=2$ and $f=(id)^r$ in Theorem \ref{TheoremEight}, we obtain \[\sum_{n_1,n_2\le x}[n_1,n_2]_b^r=C^b_{I_r,2}\frac{x^{2(r+1)}}{(r+1)^2}+\text{O}(x^{2(r+1)-\frac{1}{2}\min(r+1,1)+\epsilon}),\]
where
\begin{align*}
  C^b_{I_r,2}&= \prod_p\left(1-\frac{1}{p}\right)^2\sum_{v_1,v_2=0}^\infty \frac{p^{r\lceil\frac{\max(v_1,v_2)}{b}\rceil}}{p^{(r+1)(v_1+v_2)}}\\
  &= \prod_p\left(1-\frac{1}{p}\right)^2\left[\sum_{v_2=0}^\infty\frac{p^{r\lceil\frac{v_2}{b}\rceil}}{p^{(r+1)v_2}}\sum_{0\le v_1<v_2}\frac{1}{p^{(r+1)v_1}}+\sum_{v_2=0}^\infty\frac{1}{p^{(r+1)v_2}}\sum_{v_1\ge v_2}\frac{p^{r\lceil\frac{v_1}{b}\rceil}}{p^{(r+1)v_1}}\right]\\
  &= \prod_p\left(1-\frac{1}{p}\right)^2\left[\sum_{v_2=0}^\infty\frac{p^{r\lceil\frac{v_2}{b}\rceil}}{p^{(r+1)v_2}}\sum_{0\le v_1<v_2}\frac{1}{p^{(r+1)v_1}}+\sum_{v_1=0}^\infty\frac{p^{r\lceil\frac{v_1}{b}\rceil}}{p^{(r+1)v_1}}\sum_{0\le v_2\le v_1}\frac{1}{p^{(r+1)v_2}}\right]
  %\\
 %&= \prod_p\left(1-\frac{1}{p}\right)^2\left[\sum_{v_2=0}^\infty\frac{p^{r\lceil\frac{v_2}{b}\rceil}}{p^{(r+1)v_2}}\frac{1-\frac{1}{p^{(r+1)v_2}}}{1-\frac{1}{p^{(r+1)}}}+\sum_{v_1=0}^\infty\frac{p^{r\lceil\frac{v_1}{b}\rceil}}{p^{(r+1)v_1}}\frac{1-\frac{1}{p^{(r+1)(v_1+1)}}}{1-\frac{1}{p^{(r+1)}}}\right]\\
 \\&=\prod_p\left(1-\frac{1}{p}\right)^2\left(1-\frac{1}{p^{(r+1)}}\right)^{-1}\left[\sum_{k=0}^\infty p^{(k+1)r}\sum_{kb< v_2\le (k+1)b}\left(\frac{1}{p^{(r+1)v_2}}-\frac{1}{p^{2(r+1)v_2}}\right)\right.\\
 & \left.+\sum_{k=0}^\infty p^{(k+1)r}\sum_{kb< v_1\le (k+1)b}\left(\frac{1}{p^{(r+1)v_1}}-\frac{1}{p^{2(r+1)v_1+r+1}}\right)+\left(1-\frac{1}{p^{(r+1)}}\right)\right]\\
 &= \zeta^2(r+1)\zeta((r+1)b-r)\zeta(2(r+1)b-r)\prod_p\left(1-\frac{1}{p}\right)^2\left[\frac{1}{p}\left(1-\frac{1}{p^r}\right)\right.\\
&\left.\left(2-\frac{2}{p^{2(r+1)b-r}}-\frac{1}{p^{r+1}}+\frac{1}{p^{(r+1)b+1}}\right)+\left(1-\frac{1}{p^{2(r+1)b-r}}\right)\left(1-\frac{1}{p^{(r+1)b-r}}\right)\right].
\end{align*}  
\end{proof}
\section{Further Remarks and future Work}
Results related to the usual gcd or gcd-like functions obtained in \cite{Toth}, and the references therein, can analogously also be proved for the generalized gcd and lcm functions defined in this paper. Moreover, one can define 
variants of Ramanujan sums for generalized gcd. For example, for $j\in\mathbb{N}$, define
\[C^b_k(j):=\sum_{\substack{1\le q \le k \\ (q,k)_b=1}}\exp\left({\frac{2\pi i q j}{k}}\right).\]
Similarly, the Anderson-Apostol sums can be generalized to
\[s^b_k(j):=\sum_{d|(k,j)_b}f(d)g\left(\frac{k}{d}\right),\] for $j\in\mathbb{N}$, with arithmetic functions $f$ and $g$.
One can study the distribution of $C^b_k(j)$, and $s^b_k(j)$ by investigating its weighted mean, logarithmic mean, or moments of averages. For example, see Chan \cite{Chan}, and Robles \cite{robles}, where the authors study these aspects for Cohen sums. Ramanujan sums and their variations make appearances in singular series of the Hardy-Littlewood asymptotic formula for Waring problems and in the asymptotic formula of Vinogradov on sums of three primes. 
%One can extend the results of gcd or lcm to generalized gcd or generalized lcm, for example, many of the results stated in \cite{Toth}, and the references therein, can analogously also be proved for our functions. Moreover, similar to the Anderson-Apostol sums given by such as the results for generalized Ramanujan sums $C^b_k(z)=\sum_{\substack{1\le q \le k \\ (q,k)_b=1}}\exp({\frac{2\pi i q z}{k}})$ can be studied.
Another future direction is to study the behavior of the generalized gcd and lcm-sums along arithmetic progressions. To the best of the authors knowledge, mean value results in residue classes has not been studied even for the usual gcd and lcm-sums. Finally, one can also find the Fourier transform of the function of generalized gcd and generalized lcm.
\section*{Acknowledgements}
SC is supported by the Science and Engineering Research Board, Department of Science and Technology, Government of India under grant SB/S2/RJN-053/2018.

\end{document}